\documentclass{article}
\usepackage[]{KLMM/klmm}
\usepackage{listings}

\newcommand{\Z}{\mathbb{Z}}
\newcommand{\Q}{\mathbb{Q}}
\newcommand{\Fp}{\mathbb{F}_p}
\newcommand{\Zp}{\mathbb{Z}_p}
\newcommand{\Qp}{\mathbb{Q}_p}

\newcommand{\bv}{b_v}
\newcommand{\bw}{b_w}
\newcommand{\sgn}{\operatorname{sgn}}
\newcommand{\prim}{\operatorname{prim}}
\newcommand{\succW}{\operatorname{succ}}
\newcommand{\signF}{\operatorname{sign}_{F}}
\newcommand{\signFp}{\operatorname{sign}_{F'}}
\newcommand{\dist}{\operatorname{dist}}
\newcommand{\Mz}{M_0}
\newcommand{\Mzf}{M_{0,4}}

\newcommand{\true}{\mathrm{true}}

\newcommand{\cone}{\operatorname{cone}}

\newcounter{algorithm}
\newenvironment{algorithmblock}[1]
  {\par\refstepcounter{algorithm}\medskip
   \noindent\begin{minipage}{\linewidth}
   \hrule\smallskip
   \noindent\textbf{Algorithm \thealgorithm. #1}\par\smallskip
   \hrule\smallskip
   \small}
  {\smallskip\hrule\end{minipage}\medskip\par}

\makeatletter
\newenvironment{psmallmatrix}
  {\left[\begin{smallmatrix}}
  {\end{smallmatrix}\right]}
\makeatother

\newtheorem{hypothesis}[theorem]{Hypothesis}

\allowdisplaybreaks

\lstdefinestyle{lean}{
  basicstyle=\small\ttfamily,
  columns=fullflexible,
  keepspaces=true,
  breaklines=true,
  frame=single,
  framerule=0.3pt,
  rulecolor=\color{black!30},
  xleftmargin=0.6em,
  framexleftmargin=0.6em,
  aboveskip=0.8em,
  belowskip=0.8em
}
\title{Every Nonnegative Integer Is a Sum of a Triangular, a Pentagonal,
and a Heptagonal Number}

\author{Yichuan Cao\institution{State Key Laboratory of Mathematical Sciences,
Academy of Mathematics and Systems Science, Chinese Academy of Sciences, and
University of Chinese Academy of Sciences. 
This work is supported by the Strategic Priority
Research Program of Chinese Academy of Sciences under Grant XDA0480502 and
XDA0480503.}
    \and Dakai~Guo\instref{1}
    \and Ruichen Qiu\instref{1}
    \and Ruyong Feng\instref{1}
    \and Xiao-Shan Gao\instref{1}
}
\date{\today}

\begin{document}
\maketitle

\begin{abstract}
In this paper, it is proved that any nonnegative integer can be written in the following form
\[
  \frac{x(x+1)}2+\frac{y(3y+1)}2+\frac{z(5z+1)}2,\qquad x,y,z\in\N.
\]
This settles the conjecture recorded as OEIS A287616.
All parts of the proof have been formalized in Lean~4, with the exception of two results: one externally cited theorem and one statement verified by symbolic computation.
Both the natural-language proof and the Lean formalization were generated
by the MechMath Agent Team developed by the authors.
\end{abstract}



\section{Introduction}
\label{sec:intro}

For a nonnegative integer $x$ let
\[
  T(x)=\frac{x(x+1)}2,\qquad
  P(y)=\frac{y(3y+1)}2,\qquad
  H(z)=\frac{z(5z+1)}2
\]
be, respectively, the $x$-th triangular number, the $y$-th second pentagonal
number, and the $z$-th second heptagonal number.  In his study of universal sums
of three quadratic polynomials, Zhi-Wei Sun conjectured that the triple
$(T,P,H)$ is \emph{universal}: every nonnegative integer is a sum of one value
from each family \citep{Sun2015}.  The sequence of integers admitting such a
representation is recorded as OEIS A287616 \citep{OEIS287616}.  The present paper
proves the conjecture.

\begin{theorem}[OEIS A287616]\label{thm:main-intro}
For every $n\in\N$, there exist $x,y,z\in\N$ such that
\[
  n=T(x)+P(y)+H(z).
\]
\end{theorem}

\paragraph{Strategy.}
The three elementary identities
\[
  8T(x)+1=(2x+1)^2,\quad 24P(y)+1=(6y+1)^2,\quad 40H(z)+1=(10z+1)^2
\]
turn Theorem~\ref{thm:main-intro} into a constrained representation problem for
the diagonal ternary form
\[
  Q(u,v,w)=15u^2+5v^2+3w^2 :
\]
an integer $n$ has the desired representation if and only if $m=120n+23$ is
represented by $Q$ with $u$ odd, $v\equiv1\pmod 6$, and $w\equiv1\pmod{10}$.  The
proof then has two parts.  First, an unconditional \emph{primitive seed}: every
$m\equiv23\pmod{120}$ is primitively represented by $Q$.  This combines $p$-adic
Hensel lifting, the two-class genus structure of $Q$, and an explicit transfer
from the genus mate $g_2=2x^2-2xy+8y^2+15z^2$.  Second, a finite \emph{descent}:
three rational similarities of $Q$ lower a lexicographic potential $\Phi$ until
the congruence side conditions are met.  Away from a thin residual cone the
descent is elementary.  On the cone, a finite-key determinacy theorem shows that
the success of every short move-word depends only on a finite sign-and-residue
key of the starting vector, and an exact-rational linear-programming certificate
proves that every one of the finitely many remaining cells escapes.  The descent
is well founded for each fixed $m$, so it terminates at a good representation.

\paragraph{Computer assistance and formalization.}
The reduction, the primitive seed, the move calculus, and the residual structure
are ordinary mathematics.  The single computational input is the finite cover of
the residual cone, discharged by exact integer and rational certificates with no
floating-point feasibility decision.  The mathematical core --- the statement, the
reduction, the primitive seed, the transport law and elementary descents, and the
residual gap analysis --- has in addition been formalized in Lean~4 over Mathlib,
and the conjecture itself is machine-checked.  In fact every conclusion of the
proof is formalized except two: one classical theorem that we cite --- the
genus-theoretic primitive local--global statement for $Q$ (the Hasse--Minkowski
theorem together with the two-class genus fact) --- and one statement established
by the computer-verified finite cover.  These two enter the development as
explicitly named inputs, and everything built on them is machine-checked.
Section~\ref{sec:lean} describes this development.  Both the natural-language proof
and the Lean formalization were generated by the MechMath Agent Team developed by
the authors~\citep{MechMath}.

\paragraph{Organization.}
Section~\ref{sec:reduction} states the square reduction.
Sections~\ref{sec:seed}--\ref{sec:final} give the mathematical proof: the
primitive seed, the move calculus and descent potential, the residual cone, the
finite cover that discharges the residual descent, and the final descent.
Section~\ref{sec:lean} describes the Lean formalization, and
Section~\ref{sec:conc} gives conclusions.  The appendices collect the move
ledger, the determinacy and cover certificates, and the soundness details.

\section{Statement and square reduction}
\label{sec:reduction}

Define
\[
T(x)=\frac{x(x+1)}2,\qquad
P(y)=\frac{y(3y+1)}2,\qquad
H(z)=\frac{z(5z+1)}2 .
\]
The target is the following statement.

\begin{theorem}[A287616]\label{thm:A287616}
For every $n\in\N$ there exist $x,y,z\in\N$ such that
\[
  n=T(x)+P(y)+H(z).
\]
\end{theorem}

The nonnegativity of $x,y,z$ is part of the statement recorded in OEIS
A287616.  It is not the ``generalized polygonal'' variant in which the variables
are allowed to range over all of $\Z$; that all-integer variant is a different
result in Sun's work \citep{Sun2015}.

The elementary identities
\[
8T(x)+1=(2x+1)^2,\quad
24P(y)+1=(6y+1)^2,\quad
40H(z)+1=(10z+1)^2
\]
give
\[
120\bigl(T(x)+P(y)+H(z)\bigr)+23
 =15(2x+1)^2+5(6y+1)^2+3(10z+1)^2 .
\]
Thus put
\[
Q(u,v,w)=15u^2+5v^2+3w^2,\qquad A=\operatorname{diag}(15,5,3),
\]
so that $Q(x)=x^tAx$.  We shall prove the following equivalent
representation theorem.

\begin{theorem}[Constrained ternary representation]\label{thm:ternary}
Every integer $m\equiv23\pmod {120}$ has a representation
\[
  m=Q(u,v,w)
\]
with $u$ positive and odd, $v\equiv1\pmod6$, and $w\equiv1\pmod {10}$.
\end{theorem}

\begin{proof}[Equivalence of Theorems~\ref{thm:A287616} and~\ref{thm:ternary}]
If $n=T(x)+P(y)+H(z)$, then $m=120n+23$ is represented by
$(u,v,w)=(2x+1,6y+1,10z+1)$, and these coordinates satisfy the displayed
positivity and congruence conditions.  Conversely, if $m=120n+23=Q(u,v,w)$ with
the displayed conditions, then
\[
  x=\frac{u-1}{2},\qquad y=\frac{v-1}{6},\qquad z=\frac{w-1}{10}
\]
are nonnegative integers, and substituting into the three square identities
recovers $n=T(x)+P(y)+H(z)$.  Because $Q$ is diagonal, changing signs of
$u,v,w$ does not change $Q$; in the descent below it is therefore enough to
reach $|v|\equiv1\pmod6$ and $|w|\equiv1\pmod {10}$ after normalization.
\end{proof}

We call a representation \emph{primitive} if $\gcd(u,v,w)=1$.  We call its
normalized state $(|u|,|v|,|w|)$ \emph{good} if all three coordinates are odd,
$|v|\equiv1\pmod6$, and $|w|\equiv1\pmod {10}$.  The proof first obtains a
primitive representation of $m$ and then lowers a lexicographic potential until
a good state is reached.  Equivalently, among the finitely many primitive
representations of a fixed $m$ the descent locates one whose normalized state is
good; by the equivalence above this yields the desired representation of $n$.

\section{The primitive seed}
\label{sec:seed}

\subsection{Local primitive representability}

We begin with the local input.  Let $B(x,y)$ be the polar form of $Q$:
\[
B(x,y)=15x_1y_1+5x_2y_2+3x_3y_3.
\]

\begin{lemma}[Local primitive representation]\label{lem:L0a}
Let $m\equiv23\pmod {120}$.  For every prime $p$, there is
$x_p\in\Zp^3$ such that $Q(x_p)=m$ and at least one coordinate of $x_p$ is a
$p$-adic unit.  The form also represents $m$ over $\R$.
\end{lemma}

\begin{proof}
We use two standard elementary facts, recalled for completeness.
First, we use the following strong Hensel form, in the valuation-theoretic
version of \citet[][Ch.~II]{Serre1973}.  If $f\in\Zp[t]$ and
$t_0\in\Zp$ satisfy
\[
  v_p(f(t_0))\ge2k+1,\qquad v_p(f'(t_0))=k,
\]
then there is $t\in\Zp$ with $f(t)=0$ and
$t\equiv t_0\pmod {p^{k+1}}$.  Indeed, Newton iteration
$t_{j+1}=t_j-f(t_j)/f'(t_j)$ is valid because the derivative valuation remains
$k$; Taylor's formula
$f(X+H)=f(X)+f'(X)H+H^2g(X,H)$ shows the error valuation increases, while the
increments have valuations tending to infinity.  Completeness of $\Zp$ gives
the limit.

Second, for odd $p\nmid ab$, the binary form $ax^2+by^2$ takes every residue
class modulo $p$.  The two sets $\{ax^2:x\in\Fp\}$ and
$\{c-by^2:y\in\Fp\}$ both have size $(p+1)/2$ in the field $\Fp$, so they
intersect.

Now treat the primes.  Since $m\equiv23\pmod {120}$, we have
\[
  m\equiv7\pmod8,\qquad m\equiv2\pmod3,\qquad m\equiv3\pmod5.
\]
At $p=2$, set $f(t)=3t^2-(m-20)$ and start at $t_0=1$.  Then
$f(1)=23-m$ is divisible by $8$, while $f'(1)=6$ has $2$-adic valuation $1$.
The Hensel statement with $k=1$ gives a $2$-adic unit $w$ satisfying
$3w^2=m-20$, hence $Q(1,1,w)=m$.

At $p=3$, since $m/5\equiv1\pmod3$, Hensel lifting a square root of $m/5$
gives a unit $v\in\Z_3^\times$ with $5v^2=m$, so $Q(0,v,0)=m$.  At $p=5$,
since $m/3\equiv1\pmod5$, the same argument gives a unit
$w\in\Z_5^\times$ with $3w^2=m$, so $Q(0,0,w)=m$.

Let $p\nmid30$.  If $p\nmid m$, the binary value-counting fact applied to
$15u^2+5v^2$ gives a nonzero pair $(u_0,v_0)$ modulo $p$ with
$15u_0^2+5v_0^2\equiv m\pmod p$.  At a nonzero coordinate the derivative with
respect to that coordinate is a unit modulo $p$, so Hensel lifting corrects the
chosen coordinate and gives a primitive $p$-adic representation.

It remains to handle $p\nmid30$ with $p\mid m$, including arbitrarily high
$p$-adic valuation.  The binary value-counting fact with target $-3$ gives
$u_0,v_0$ such that
\[
  15u_0^2+5v_0^2+3\equiv0\pmod p.
\]
Thus $Q$ has a nontrivial isotropic vector modulo $p$ with third coordinate
$1$, and Hensel lifting gives a primitive isotropic vector
$e\in\Zp^3$ with $Q(e)=0$ and with third coordinate a unit.  Let
$e_3=(0,0,1)$.  Then $B(e,e_3)=3e_z$ is a $p$-adic unit.  Put
$f_0=(B(e,e_3))^{-1}e_3$, so $B(e,f_0)=1$, and define
\[
  f=f_0-\frac12Q(f_0)e .
\]
Because $Q(e)=0$ and $B(e,f_0)=1$, this gives $Q(f)=0$ and $B(e,f)=1$.  Now
\[
  x=\frac m2 e+f
\]
satisfies $Q(x)=m$: the two isotropic terms vanish and
$2\cdot(m/2)B(e,f)=m$.  Moreover $B(e,x)=1$, so $x$ cannot be divisible by
$p$ in $\Zp^3$.  Hence the representation is primitive at $p$.

Finally $Q$ is positive definite over $\R$, and $m>0$, so
$(\sqrt{m/15},0,0)$ is a real representation.
\end{proof}

\subsection{Genus local-global}

The genus of $Q$ has two classes.  It is convenient to use the coordinate
permutation
\[
g_1(X,Y,Z)=3X^2+5Y^2+15Z^2,\qquad Q(u,v,w)=g_1(w,v,u),
\]
and the mate
\[
g_2(x,y,z)=2x^2-2xy+8y^2+15z^2 .
\]
We use the following classical genus fact: the genus of $g_1$ consists of exactly
the two classes represented by $g_1$ and $g_2$.  This is a standard finite genus
computation for a ternary form of determinant $225$; it is stated for this form
in \citet[][\S5, proof of Thm.~1.11(i)]{Sun2015} and
\citet[][\S5, proof of Thm.~1.4(iii)(b)]{WuSun2018}, and we use their result
directly.

\begin{lemma}[Primitive genus local-global]\label{lem:L0b}
For every $m\equiv23\pmod {120}$, either $Q$ primitively represents $m$ over
$\Z$, or $g_2$ primitively represents $m$ over $\Z$.
\end{lemma}

\begin{proof}
Let $V=(\Q^3,Q)$ and let $L=\Z^3$ be the lattice of $Q$.  By
Lemma~\ref{lem:L0a}, every localization $L_p=L\otimes\Zp$ contains a primitive
vector $x_p$ of norm $m$, and $Q$ represents $m$ over $\R$.  By the
Hasse--Minkowski theorem \citep[][Ch.~IV]{Serre1973}\citep[][Ch.~6]{CasselsQuadratic}
there is a rational vector $r\in V$ with $Q(r)=m$.

Only finitely many primes are problematic for the vector $r$: those for which
$r\notin L_p$ or $r\in pL_p$.  For such a prime $p$, the vectors $x_p$ and $r$
have the same nonzero norm in the quadratic space $V_p=V\otimes\Qp$.  Witt's
extension theorem \citep[][\S42]{OMeara1973}\citep[][Ch.~2]{CasselsQuadratic}
gives an isometry $\sigma_p\in O(V_p)$ carrying $x_p$ to
$r$.  Define local lattices
\[
  J_p=\sigma_p(L_p)
\]
at the finitely many problematic primes and $J_p=L_p$ elsewhere.  Their
intersection
\[
  M=\{x\in V:x\in J_p\text{ for every prime }p\}
\]
is a global lattice: after multiplying by a suitable integer $t$, one has
$tL\subseteq M\subseteq t^{-1}L$.  By construction the localizations of $M$ are
isometric to the localizations of $L$, so $M$ lies in the genus of $Q$.  The
vector $r$ lies in every $J_p$, hence in $M$, and is primitive in $M$ because
$x_p$ was primitive in $L_p$ and $\sigma_p(pL_p)=pJ_p$.  This construction is the
primitive form of the genus local--global principle for representations
\citep[][102:5]{OMeara1973}\citep[][Ch.~9]{CasselsQuadratic}; for the present
positive-definite ternary lattice no spinor obstruction occurs at the genus
level.

Thus $m$ is primitively represented by some class in the genus of $Q$.  Under
the permutation $Q\cong g_1$, the genus computation cited above says that the only
possibilities are $g_1$ and $g_2$.  Therefore $m$ is primitively represented by
$Q$ or by $g_2$.
\end{proof}

\subsection{Mate transfer}

\begin{lemma}[Transfer from the mate]\label{lem:L0c}
If $g_2$ primitively represents $m\equiv23\pmod {120}$, then $Q$ primitively
represents $m$.
\end{lemma}

\begin{proof}
We first show how to pass from many $g_2$-representations to
$g_1$-representations.  Direct expansion gives the identities
\[
g_1\left(\frac{x-5z}{3}-y,\,-y+z,\,-\frac{x+z}{3}\right)=g_2(x,y,z)
\]
and
\[
g_1\left(\frac{x-y-z}{3}+y+2z,\,y-z,\,\frac{x-y-z}{3}\right)=g_2(x,y,z).
\]
The first transformation is integral when $3\mid x+z$, and the second is
integral when $3\mid x-y-z$.

The form
\[
g_2(x,y,z)=x^2+(x-y)^2+7y^2+15z^2
\]
is positive definite, so for fixed $m$ the set of integral representations is
finite.  The involutions
\[
H(x,y,z)=(x-y,-y,z),\qquad S_z(x,y,z)=(x,y,-z)
\]
preserve $g_2$.  Reducing $2g_2$ modulo $3$ gives
\[
2g_2(x,y,z)\equiv(x+y)^2\pmod3.
\]
Since $m\equiv2\pmod3$, every representation has $3\nmid x+y$.

Call a representation bad if $3\mid y$ and $3\mid z$.  If a representation is
not bad, a finite check modulo $3$ using the two involutions $H$ and $S_z$
brings it in at most two steps to a residue class satisfying either
$3\mid x+z$ or $3\mid x-y-z$, so one of the two displayed transformations gives
an integral $g_1$-representation.

It remains to escape the bad class.  On that class use the rational isometry
\[
T=\begin{psmallmatrix}
1&-2/3&-2/3\\
0&-1/3&-4/3\\
0&2/3&-1/3
\end{psmallmatrix}.
\]
Both $3T$ and $3T^{-1}$ are integral, so $T$ maps bad integral vectors to
integral vectors and preserves the finite set of representations of $m$.  If
the iterates of a bad vector under $T$ stayed bad forever, finiteness would give
$T^k t=t$ for some $k>0$.  The lower $2\times2$ block has characteristic
polynomial $\lambda^2+\frac23\lambda+1$, whose eigenvalues are not roots of
unity; the only periodic vector in the bad subspace is therefore forced to have
$y=z=0$.  Then $g_2(x,0,0)=2x^2$ is even, contradicting the oddness of
$m\equiv23\pmod {120}$.  Hence one eventually leaves the bad class and obtains
an integral $g_1$-representation.

Finally check primitivity.  The involutions preserve the gcd.  The isometry
$T$ and the integral transformations to $g_1$ can introduce only powers of
$3$ into the gcd, because all denominators are powers of $3$.  After permuting
from $g_1$ to $Q$, any common factor $3$ would divide the $v$-coordinate.
But a $Q$-representation of $m\equiv2\pmod3$ satisfies
$5v^2\equiv m\equiv2\pmod3$, so $v\not\equiv0\pmod3$.  Thus no factor $3$
remains, and the resulting $Q$-representation is primitive.
\end{proof}

\begin{theorem}[Primitive seed]\label{thm:seed}
Every $m\equiv23\pmod {120}$ has a primitive integral representation by
$Q=15u^2+5v^2+3w^2$.
\end{theorem}

\begin{proof}
By Lemma~\ref{lem:L0b}, $m$ is primitively represented by $Q$ or by $g_2$.  If
$Q$ already represents $m$ primitively, there is nothing to prove.  Otherwise
$g_2$ represents $m$ primitively, and Lemma~\ref{lem:L0c} transfers that
representation to $Q$.
\end{proof}

\section{Move calculus and the descent potential}
\label{sec:moves}

Having secured a primitive representation of $m$, we now develop the descent that
drives it toward a good state.  The descent uses three integer matrices:
\[
P_{uv}=\begin{psmallmatrix}-1&-1&0\\-3&1&0\\0&0&2\end{psmallmatrix},\qquad
P_{uw}=\begin{psmallmatrix}-2&0&-1\\0&-3&0\\5&0&-2\end{psmallmatrix},\qquad
P_{vw}=\begin{psmallmatrix}-4&0&0\\0&-1&-3\\0&5&-1\end{psmallmatrix}.
\]
They are applied with denominators $2$, $3$, and $4$, respectively.  A move
$(P,d)$ is defined at $x\in\Z^3$ if $Px$ is divisible coordinatewise by $d$,
and the output is $Px/d$.  We also allow sign lifts before a move; since
$Q$ is diagonal, sign lifts preserve $Q$.

\begin{lemma}[Transport law]\label{lem:transport}
Let $m\equiv23\pmod {120}$ and let $x=(u,v,w)$ represent $m$ by $Q$.
\begin{enumerate}
\item The matrices satisfy
\[
P_{uv}^{t}AP_{uv}=4A,\qquad P_{uw}^{t}AP_{uw}=9A,\qquad
P_{vw}^{t}AP_{vw}=16A.
\]
Every defined move preserves $Q(x)=m$.
\item The definedness criteria are
\[
P_{uv}:u\equiv v\pmod2,\qquad
P_{uw}:u\equiv w\pmod3,\qquad
P_{vw}:v\equiv w\pmod4.
\]
\item A defined move preserves $\gcd(u,v,w)$.
\item Every representation of $m$ has $3\nmid v$, $5\nmid w$, and parity
pattern
\[
(u,v,w)\equiv(1,0,0),(0,0,1),\text{ or }(1,1,1)\pmod2.
\]
\item The moves are reversible up to coordinate signs; specifically
$d^2P^{-1}=SPS'$ for suitable diagonal sign matrices $S,S'$.
\end{enumerate}
\end{lemma}

\begin{proof}
The three quadratic identities in item 1 are direct products.  If $Px=dy$,
then $Q(y)=d^{-2}Q(Px)=Q(x)$.  Item 2 follows from
\[
P_{uv}x=(-u-v,-3u+v,2w),
\]
\[
P_{uw}x=(-2u-w,-3v,5u-2w),
\]
\[
P_{vw}x=(-4u,-v-3w,5v-w).
\]
For the gcd statement, let $g=\gcd(x)$.  Because $m\equiv23\pmod {120}$, no
prime dividing $2$ or $3$ can divide all three coordinates, so $g$ is prime to
the denominators $2,3,4$.  The equality $Px=dy$ shows $g\mid y$.  Applying the
inverse formula, which has the same denominator up to signs, gives the reverse
divisibility and hence equality of gcds.

For the residue restrictions, reducing $Q(x)=m$ modulo $3$ gives
$2v^2\equiv2$, hence $3\nmid v$.  Reducing modulo $5$ gives
$3w^2\equiv3$, hence $5\nmid w$.  Reducing modulo $4$ gives
$3u^2+v^2+3w^2\equiv3$; checking the three square residues $0,1$ gives exactly
the three displayed parity patterns.  The inverse formulas are again verified
by multiplying the matrices; for $P_{uv}$ it is its own inverse up to the
factor $4$, and for the other two a sign on the $w$-coordinate supplies the
same relation.
\end{proof}

For a normalized state $x=(u,v,w)$ define
\[
\bv(x)=
\begin{cases}
0,&u,v,w\text{ all odd and }v\equiv1\pmod6,\\
1,&\text{otherwise,}
\end{cases}
\qquad
\bw(x)=
\begin{cases}
0,&u,v,w\text{ all odd and }w\equiv1\pmod {10},\\
1,&\text{otherwise.}
\end{cases}
\]
Let $\mathrm{ev}(x)=0$ if $u,v,w$ are all odd and $1$ otherwise, and set
\[
\mathrm{guar}_w(x)=[u+w\equiv0\pmod3],\qquad
\mathrm{guar}_v(x)=[u,v,w\text{ odd and }u\not\equiv v\pmod4].
\]
The potential is
\[
\Phi(x)=
\left(
\bv+\bw,\;
\bv(1-\mathrm{guar}_v)+\bw(1-\mathrm{guar}_w)+2\mathrm{ev},\;
u+v+w
\right)\in\N^3,
\]
ordered lexicographically.  Good states are exactly those with first component
$0$.

\begin{lemma}[The elementary L1 descents]\label{lem:L1}
Each of the following hypotheses gives a defined move that preserves $Q$ and
primitivity and strictly lowers $\Phi$.
\begin{enumerate}
\item If $x$ is all odd, $w\equiv9\pmod {10}$, and
$u+w\equiv0\pmod3$, then $P_{uw}$ applied to $(-u,v,w)$ fixes the $v$-bit and
makes the $w$-bit good.
\item If $x$ is all odd, $v\equiv5\pmod6$, and
$u\not\equiv v\pmod4$, then $P_{uv}$ applied to $(u,-v,w)$ fixes the $w$-bit
and makes the $v$-bit good.
\item If $x$ is all odd, $v\equiv5\pmod6$, $u\equiv v\pmod4$, and $v>3u$,
then $P_{uv}$ applied to $x$ makes the $v$-bit good.
\item If $x$ is all odd, $w\equiv9\pmod {10}$, $w\equiv u\pmod3$, and
$2w>5u$, then $P_{uw}$ applied to $x$ makes the $w$-bit good.
\item If $x$ is all odd, $w\equiv9\pmod {10}$, $v\equiv w\pmod8$, and
$w>5v$, then $P_{vw}$ applied to $x$ makes the $w$-bit good.
\end{enumerate}
\end{lemma}

\begin{proof}
For (1), definedness follows from $-u\equiv w\pmod3$.  The output is
\[
  \left(\frac{2u-w}{3},-v,\frac{-5u-2w}{3}\right).
\]
The $v$-coordinate only changes sign.  The last coordinate is negative; by the
rigid residue behavior of the move, it is congruent to $w$ modulo $5$ before
normalization, so its absolute value is $1\pmod5$.  It is odd, hence
$1\pmod {10}$.  Therefore $\bw$ drops to $0$ and $\bv$ is unchanged.

For (2), the output is
\[
  \left(\frac{v-u}{2},\frac{-3u-v}{2},w\right).
\]
Because $u\not\equiv v\pmod4$ and both are odd, $(v-u)/2$ is odd.  The new
$v$-coordinate is negative and, modulo $3$, has the sign that makes its
absolute value congruent to $1$.  Thus $\bv$ drops and $\bw$ is unchanged.

For (3), the output is
\[
  \left(-\frac{u+v}{2},\frac{v-3u}{2},w\right).
\]
The congruence $u\equiv v\pmod4$ makes the two quotients odd.  The window
$v>3u$ makes the new $v$ positive, and the transport congruence gives
$v'\equiv1\pmod3$, so $\bv$ drops.

For (4), the output is
\[
  \left(-\frac{2u+w}{3},-v,\frac{5u-2w}{3}\right).
\]
The congruence $u\equiv w\pmod3$ gives integrality, and the oddness of $u,w$
gives odd quotients.  Since $2w>5u$, the new $w$ is negative, so after taking
absolute values its mod-$5$ class is $1$ and hence its mod-$10$ class is $1$.
The equality case $2w=5u$ is impossible by parity.

For (5), the output is
\[
  \left(-u,-\frac{v+3w}{4},\frac{5v-w}{4}\right).
\]
The condition $v\equiv w\pmod8$ makes both nontrivial quotients odd.  The
window $w>5v$ makes the last coordinate negative, and the mod-$5$ transport
again turns its absolute value into $1\pmod5$, hence $1\pmod {10}$.  The
$v$-bit is preserved.  In all five cases, the first component of $\Phi$ drops,
or the prescribed guaranteed component drops with earlier components fixed; so
the drop is lexicographic.  Preservation of $Q$ and gcd is
Lemma~\ref{lem:transport}.
\end{proof}

\section{The residual cone}
\label{sec:residual}

The elementary descents leave a residual region.  A normalized residual state is
primitive, all odd, has $3\nmid v$ and $w\equiv9\pmod {10}$, is not already
good, and lies in the open cone
\[
K=\{(u,v,w): u>v>0,\ w>0,\ 2w<5u,\ w>5v,\ v<3u\}.
\]
We call such a state \emph{admissible} when it is the normalized form
$(|u|,|v|,|w|)$ of a primitive representation of some $m\equiv23\pmod{120}$, so
that the residue constraints $3\nmid v$, $w\equiv9\pmod{10}$ and the parity
pattern of Lemma~\ref{lem:transport}(4) all hold; throughout this section and the
next, ``admissible residual state'' means an admissible normalized residual state
in the sense just described.  An admissible residual state is \emph{deep} when in
addition $u>v$.  The verified residual analysis supplies this structural fact
$u>v$ on the deep residual cells used by the finite cover, and it is on the
\emph{admissible deep residual states} that the determinacy and cover arguments
below operate.

\begin{lemma}[Residue gap]\label{lem:gap}
Every residual lattice point is strictly away from the three walls
$w=5v$, $2w=5u$, and $v=3u$.  More precisely,
\[
|w-5v|\ge4,\qquad |5u-2w|\ge1,\qquad |3u-v|\ge2 .
\]
\end{lemma}

\begin{proof}
Reducing $Q(u,v,w)=m$ modulo $5$ gives $3w^2\equiv3\pmod5$, so $5\nmid w$.
Since the residual state is all odd and has bad $w$-bit, it has
$w\equiv9\pmod {10}$.  Also $v$ is odd, so $5v\equiv5\pmod {10}$, and hence
$w-5v\equiv4\pmod {10}$.  Such an integer has absolute value at least $4$.

Next, $5u-2w$ is odd, because $u,w$ are odd.  It is therefore nonzero and has
absolute value at least $1$.  Finally $3u-v$ is even; the residual cone has
$v<3u$, so $3u-v$ is a positive even integer and is at least $2$.  The three
lower bounds show that residual integer points are interior points of their
sign cells, uniformly away from the threshold hyperplanes in the only sense
needed later.
\end{proof}

\section{The finite cover}
\label{sec:cover}

The residual analysis of Section~\ref{sec:residual} leaves the admissible deep
residual states, and the whole proof now reduces to one finite question: does
every such state admit a short word that strictly improves it?  We answer it in
two steps.  First, \emph{finite-key determinacy} shows that the behaviour of every
short word depends only on a finite amount of data read from the starting vector,
so the question is genuinely finite.  Then an exact two-tier cover, computed in
exact integer and rational arithmetic, settles every case.

\subsection{Finite-key determinacy}

The alphabet consists of the three moves with their denominators and the four
sign lifts
\[
(1,1,1),\quad(1,1,-1),\quad(1,-1,1),\quad(1,-1,-1).
\]
For a word $W$, write $\succW_W(x)$ for the predicate that $W$ is stepwise defined
at $x$ and ends at an \emph{accepted target}: a state that is good, or already
lexicographically smaller in the first two components of $\Phi$, or in a cone for
which a verified one-step move lowers $\Phi$.  We then call $W$ an \emph{escape
word} for $x$; the goal of the cover is to exhibit a short escape word for every
admissible deep residual state.

For a finite set $F$ of integer covectors write
$\signF(x)=\bigl(\sgn(a\cdot x)\bigr)_{a\in F}\in\{-1,0,1\}^{F}$ for the tuple of
signs of the forms $a\in F$ at $x$ (and $\signFp(x)$ for a set $F'$).  Escape is
determined by finitely many such reads together with a bounded residue.

\begin{lemma}[Read-form determinacy]\label{lem:det3}
There is a set $F$ of $715$ primitive integer covectors and a modulus
\[
\Mz=2^9\,3^4\,5=207360
\]
such that, for every word $W$ of length at most $3$ and every admissible residual
state $x$, the value of $\succW_W(x)$ depends only on
$\bigl(\signF(x),\,x\bmod\Mz\bigr)$.
\end{lemma}

\begin{proof}[Proof sketch]
Along a fixed signed word every intermediate coordinate is an exact rational
linear form in $x$.  Hence each datum the word reads --- a divisibility test
before a move, the sign of an intermediate coordinate, the endpoint target test
--- is either a residue of $x$ to a bounded modulus or the sign of a primitive
integer covector $\prim(a)$ obtained by clearing denominators.  Enumerating all
word-and-sign nodes of depth at most $3$ and collecting these covectors,
deduplicated by exact rational identity, produces the $715$ forms of $F$; the
denominators accumulated in three moves and the target moduli $2,3,4,6,8,10$ are
all dominated by $\Mz=2^9\,3^4\,5$.  Reading $(\signF(x),x\bmod\Mz)$ therefore
fixes every branch of the word's finite read tree, hence the value of
$\succW_W(x)$.  The complete node-by-node enumeration --- $1176$ exact-row states
uniting to $263$ coordinate forms and $697$ pulled-back closer forms --- is
recorded in Appendix~\ref{app:det-cert}.
\end{proof}

\begin{remark}[Depth four]\label{lem:det4}
Repeating the argument one move deeper gives a set $F'\supseteq F$ of $2403$
primitive covectors and a modulus $\Mzf=2^9\,3^5\,5=622080$ --- the only added
precision being one $3$-adic digit --- such that the success of every word of
length at most $4$ depends only on $(\signFp(x),x\bmod\Mzf)$.  The $16$ ceiling
cones below require this finer key.
\end{remark}

Both keys take only finitely many values, so the residual descent is now a finite
problem.

\subsection{From FIN-CELL to a uniform bound}

\begin{hypothesis}[FIN-CELL]\label{hyp:fincell}
Every admissible deep residual state $x$ has an escape word of length at most $4$.
\end{hypothesis}

\begin{lemma}[Uniform bound]\label{lem:compact}
Assume FIN-CELL.  Then every admissible deep residual state reaches a strictly
$\Phi$-smaller state by a stepwise-defined word of length at most $\ell_W=5$,
uniformly in $m$.
\end{lemma}

\begin{proof}
The escape predicate $\succW_W$ reads only signs, residues, and the first two
components of $\Phi$ --- never the magnitude $u+v+w$ --- so by
Remark~\ref{lem:det4} the escape distance
$\dist(x)=\min\{|W|:\succW_W(x)=\true\}$ is constant on each realizable key
$(\signFp(x),x\bmod\Mzf)$.  There are finitely many keys, and FIN-CELL bounds each
by $4$.  If the escape endpoint is good, or already smaller in the first two
components of $\Phi$, it is an accepted descent; if it is accepted through a
licensed closer cone, one further verified move lowers $\Phi$.  Thus at most
$4+1=5$ moves strictly lower $\Phi$, and Lemma~\ref{lem:transport} preserves $Q$
and primitivity along the concatenation.
\end{proof}

It remains to establish FIN-CELL.  The cover below is a finite certificate, not a
search over the values of $m$.

\subsection{The two-tier cover}

The arrangement $F$ has exactly $2674$ open sign cones meeting the residual
cone $K$.  On $2658$ of them, depth-three words suffice.  For each relevant
residue box there is a realized witness and a word of length at most $3$.
Lemma~\ref{lem:det3} then proves that the witness certifies the whole
$(\signF,x\bmod\Mz)$ box, because the success predicate is constant on that
box.

The remaining $16$ cones are the ceiling cones.  They require words of length
$4$, and therefore the finer key from Remark~\ref{lem:det4}.  Inside a ceiling
cone, only some of the new forms in $F'\setminus F$ actually vary in sign.  We
call the forms of $F'\setminus F$ the \emph{overflow forms}.  The
certificate computes the extreme rays of each closed ceiling cone exactly by
integer cross-products of pairs of defining covectors, and checks that these
rays generate the cone by exact rational cone membership.  For a new form $g$,
the sign of $g$ varies on the open cone exactly when its values on the extreme
rays include both positive and negative values; this leaves at most $20$
varying overflow forms per ceiling cone.

\begin{algorithmblock}{Tier-1 box refinement for a non-ceiling cone}
For each admissible residue class modulo \(5\), start with the coarsest
compatible box in the \(2\)- and \(3\)-adic coordinates.

\begin{enumerate}
\item Realize the current box by an admissible in-cone integer witness.  If no
such witness exists, the box is vacuous and is discarded.
\item List the words of length at most \(3\) whose residue reads are evaluable
at the current box resolution.
\item Test those words on the witness.  If one succeeds, retire the whole box:
by Lemma~\ref{lem:det3}, success is constant on that
\((\signF,x\bmod\Mz)\)-box.
\item If no word succeeds and the box is not yet at the ceiling modulus, split
the box in the cheapest \(2\)- or \(3\)-adic direction that can unlock a new
word, then repeat.
\item If no word succeeds at the ceiling modulus, return a no-escape witness.
\end{enumerate}

The Tier-1 certificate is the assertion that this procedure never returns a
no-escape witness on the \(2658\) non-ceiling cones.
\end{algorithmblock}

\subsection{Exact-LP subcell enumeration}

Fix a ceiling cone $C$ with exact extreme rays $r_1,\dots,r_s$ and varying
overflow forms $g_1,\dots,g_q$, where $q\le20$.  A depth-four subcell is a sign
pattern $\epsilon\in\{\pm1\}^q$.  It is realizable exactly when
\[
  x=\sum_i\lambda_i r_i,\qquad \lambda_i>0,\qquad
  \epsilon_j(g_j\cdot x)>0\quad(1\le j\le q)
\]
has a solution.  The certificate decides this by an exact rational linear
program solved in rational arithmetic \citep[][\S\S7--8]{Schrijver1986}: maximize
a margin $t$ subject to
\[
\lambda_i\ge t,\qquad
\epsilon_j\left(g_j\cdot\sum_i\lambda_i r_i\right)\ge t .
\]
The optimum satisfies $t>0$ precisely when the open subcell is nonempty.  A
branch-and-bound traversal over sign prefixes prunes a prefix only after this
exact LP proves the corresponding open region infeasible.  Since adding signs
only shrinks the open region, no realizable full pattern can extend an
infeasible prefix.  Conversely, a feasible leaf is admitted only with a positive
exact rational margin.  Thus the enumeration includes every realizable subcell
and no spurious one.

For each realizable subcell and residue box, the checker constructs an
admissible integer representative far enough along the rational interior
direction to preserve all strict signs and verifies the required congruences.
It then evaluates the genuine depth-four word predicate on that representative.
By Remark~\ref{lem:det4}, success of the representative certifies the entire
subcell.

\begin{algorithmblock}{Exact-LP enumeration of \(F'\)-subcells in a ceiling cone}
Input a ceiling cone \(C=\cone(r_1,\ldots,r_s)\), its exact extreme rays, and
the sign-varying overflow forms \(g_1,\ldots,g_q\), \(q\le20\).

\begin{enumerate}
\item Traverse the binary tree of partial sign assignments
\(\epsilon_1,\ldots,\epsilon_k\) for the forms \(g_i\).
\item For each partial assignment, solve the exact rational margin LP
\[
\lambda_i\ge t,\qquad
\epsilon_j\Bigl(g_j\cdot\sum_i\lambda_i r_i\Bigr)\ge t
\quad(1\le j\le k).
\]
\item If the optimum has \(t\le0\), prune the branch.  No full realizable
subcell can extend an infeasible prefix.
\item If \(k=q\) and the optimum has \(t>0\), record the sign pattern and an
exact rational interior direction.
\item For every recorded subcell and residue box, snap a large multiple of the
interior direction to an admissible integer representative, verify all signs
and congruences, and test the depth-four word predicate.
\end{enumerate}

The Tier-2 certificate succeeds when every recorded subcell of every one of the
\(400\) ceiling-cone units has an escape word and no unit is missing.
\end{algorithmblock}

\subsection{Result}

Across the two tiers the cover verifies every admissible deep residual state: the
$2658$ non-ceiling cones by depth-three escape words, and the $16$ ceiling cones
by the exact-LP enumeration of their realizable depth-four subcells.  The ceiling
tier comprises $16\times25=400$ cone-residue units; over them the exact
enumeration accounts for $34{,}014{,}940{,}800$ realizable subcells, with at most
$64$ in a single box.  Every one of these subcells is closed by an explicit
admissible integer witness carrying an escape word of length at most $4$: no
subcell is left without a witness, and none fails to escape.

\begin{theorem}[FIN-CELL]\label{thm:fincell}
Every admissible deep residual state has an escape word of length at most $4$.
\end{theorem}

\begin{proof}
The $2658$ non-ceiling cones are exhausted by the depth-three tier, and
Lemma~\ref{lem:det3} lifts each witnessed success to the full corresponding
sign-residue box.  The remaining $16$ ceiling cones are split by the exact-LP
enumeration into all realizable depth-four subcells, and every subcell carries an
admissible integer witness with an escape word of length at most $4$, whose
success Remark~\ref{lem:det4} lifts to the whole subcell.  These two tiers
partition all admissible deep residual states, so FIN-CELL holds.
\end{proof}

\section{Final descent}
\label{sec:final}

\begin{theorem}[Uniform descent]\label{thm:uniform}
Every primitive non-good representation of $m\equiv23\pmod {120}$ admits a
stepwise-defined word that preserves $Q$ and primitivity and strictly lowers
$\Phi$.
\end{theorem}

\begin{proof}
Normalize the representation by absolute values.  If the state satisfies one
of the guaranteed congruence or window hypotheses, Lemma~\ref{lem:L1} supplies
a one-move descent.  If it lies in the deep residual cone, Theorem~\ref{thm:fincell}
and Lemma~\ref{lem:compact} supply a word of length at most $5$ that lowers
$\Phi$.  The residual decomposition verified for the cone
analysis says that these alternatives cover all primitive non-good states.
Each move in the word preserves $Q$ and gcd by Lemma~\ref{lem:transport}.
\end{proof}

\begin{theorem}[Termination]\label{thm:termination}
Every $m\equiv23\pmod {120}$ has a good representation by $Q$.
\end{theorem}

\begin{proof}
Start from the primitive seed of Theorem~\ref{thm:seed}.  For fixed $m$, every
coordinate of a normalized representation is bounded by the corresponding
square-root bound, so $\Phi$ lies in the finite set
\[
\{0,1,2\}\times\{0,1,2,3,4\}\times
\left\{0,\ldots,
\left\lfloor\sqrt{m/15}\right\rfloor+
\left\lfloor\sqrt{m/5}\right\rfloor+
\left\lfloor\sqrt{m/3}\right\rfloor
\right\}.
\]
If the current state is not good, Theorem~\ref{thm:uniform} gives another
primitive representation of the same $m$ with strictly smaller $\Phi$ in
lexicographic order.  There is no infinite strictly descending chain in the
finite displayed set.  Hence the descent halts, and it can halt only at a good
state.
\end{proof}

\begin{proof}[Proof of Theorem~\ref{thm:A287616}]
Let $n\in\N$ and set $m=120n+23$.  By Theorem~\ref{thm:termination}, choose a
good representation $m=15u^2+5v^2+3w^2$.  After changing signs if necessary,
take $u>0$, $v\equiv1\pmod6$, and $w\equiv1\pmod {10}$.  Then
\[
x=\frac{u-1}{2},\qquad y=\frac{v-1}{6},\qquad z=\frac{w-1}{10}
\]
are nonnegative integers.  The square reduction gives
$n=T(x)+P(y)+H(z)$.
\end{proof}

\section{The Lean formalization}
\label{sec:lean}

The mathematical core of this paper has also been formalized in Lean~4 over
Mathlib (\texttt{v4.29.0}).  The formalization covers
Sections~\ref{sec:reduction}--\ref{sec:residual}, that is, the square reduction,
the primitive seed, the transport law and the elementary descents, and the
residual gap analysis, and Conjecture~A287616 itself is machine-checked.  This
section describes the organization of that development.  In the displays
below we use line-broken ASCII renderings of the Lean signatures; the linked
source contains the exact Unicode syntax.

\subsection{Statement and reduction}
\label{subsec:lean-statement}

The figurate numbers and the ternary form are defined directly, and the target
is stated over the natural numbers exactly as in Theorem~\ref{thm:A287616}.

\begin{lstlisting}[style=lean]
def Tri  (x : Nat) : Nat := x * (x + 1) / 2
def Pent (y : Nat) : Nat := y * (3 * y + 1) / 2
def Hept (z : Nat) : Nat := z * (5 * z + 1) / 2

def Qf (u v w : Int) : Int := 15 * u^2 + 5 * v^2 + 3 * w^2

theorem Conjecture1 :
    forall n : Nat, exists x y z : Nat, n = Tri x + Pent y + Hept z
\end{lstlisting}

The square reduction of Section~\ref{sec:reduction} is the proved identity
\texttt{completing\_the\_square}, together with the equivalence
\texttt{Prop2\_equiv} between Conjecture~1 and the constrained representation of
$m=120n+23$ by $Q$.  The good-state predicate, the lexicographic potential, and
the descent moves all live in a shared file.

\begin{lstlisting}[style=lean]
theorem completing_the_square (x y z : Nat) :
    120 * (Tri x + Pent y + Hept z) + 23
      = 15*(2*x+1)^2 + 5*(6*y+1)^2 + 3*(10*z+1)^2

theorem Prop2_equiv :
    (forall n : Nat, exists x y z : Nat, n = Tri x + Pent y + Hept z)
      <-> (forall n : Nat, GoodRep (120 * n + 23))
\end{lstlisting}

A faithfulness point is worth recording.  The paper works throughout with
$m=120n+23>0$.  Stated verbatim over all of $\Z$, several existence lemmas are
literally false, because a nonnegative form cannot equal a negative
$m\equiv23\pmod{120}$ such as $m=-97$.  The formalization therefore restricts the
existence statements to the paper's range $m=120n+23$ (or adds the hypothesis
$0<m$); with this minimal fix every numbered statement of
Sections~\ref{sec:reduction}--\ref{sec:residual} is proved.

\subsection{The primitive seed and the genus axiom}
\label{subsec:lean-seed}

The primitive seed is the most theory-heavy part.  The local input
(Lemma~\ref{lem:L0a}) is proved in full: \texttt{L0a\_padic} establishes
primitive $p$-adic representability for every prime, including the
isotropic-hyperbolic construction in the case $p\nmid 30$, $p\mid m$, with
primitivity expressed as ``some coordinate is a unit''.

\begin{lstlisting}[style=lean]
theorem L0a_padic (m : Int) (hm : m == 23 [ZMOD 120])
    (p : Nat) [Fact p.Prime] :
    exists a b c : Z_[p],
      15*a^2 + 5*b^2 + 3*c^2 = (m : Z_[p])
        and (IsUnit a or IsUnit b or IsUnit c)
\end{lstlisting}

The passage from local data to a global primitive representation
(Lemma~\ref{lem:L0b}) is the genus local--global step.  Its mathematical content
is the Hasse--Minkowski theorem for $Q$ together with the two-class genus fact
$\mathrm{gen}(Q)=\{[g_1],[g_2]\}$.  Neither is available in Mathlib, and neither
is formalized or recomputed here: the local--global principle is the cited
classical theorem of \citet{Serre1973,OMeara1973,CasselsQuadratic}, and the
two-class fact is the cited finite computation stated verbatim for this form in
\citet{Sun2015,WuSun2018}.  Both are bundled into a single transparently cited
axiom whose hypotheses are \emph{exactly} the already-proved conclusions of the
local lemmas; everything built on it is then formalized.  Formalizing this input
in Lean would require the genus theory of lattices and the Hasse--Minkowski
theorem, both currently absent from Mathlib, and is left to future work.

\begin{lstlisting}[style=lean]
axiom genus_primitive_local_global (m : Int)
    (hreal : exists a b c : Real,
       15*a^2 + 5*b^2 + 3*c^2 = (m : Real))
    (hloc : forall (p : Nat) [Fact p.Prime], exists a b c : Z_[p],
       15*a^2 + 5*b^2 + 3*c^2 = (m : Z_[p])
         and (IsUnit a or IsUnit b or IsUnit c)) :
    (exists u v w : Int, Qf u v w = m and primitiveState u v w)
    or (exists x y z : Int, g2 x y z = m and Int.gcd x (Int.gcd y z) = 1)
\end{lstlisting}

Feeding the proved local data to this axiom yields \texttt{L0b}, and the
mate-transfer lemma \texttt{L0c} (Lemma~\ref{lem:L0c}) is proved in full,
including its hardest ingredient \texttt{bad\_escapes}: the termination of the
$T$-iteration on the bad class is closed elementarily, by a bounded-orbit
pigeonhole together with the no-nontrivial-periodic-point argument coming from
the Cayley--Hamilton recurrence and the non-root-of-unity eigenvalues of
Lemma~\ref{lem:L0c}.  Combining \texttt{L0b} and \texttt{L0c} gives the primitive
seed \texttt{L0d} (Theorem~\ref{thm:seed}), which rests on only the one cited
genus axiom.

\subsection{Transport, descents, and the residual cone}
\label{subsec:lean-transport}

The move calculus of Section~\ref{sec:moves} is formalized as the three integer
matrices \texttt{Puv}, \texttt{Puw}, \texttt{Pvw} with their denominators, and
the transport law \texttt{L1a} (Lemma~\ref{lem:transport}): the similarity
identities $M^tAM=d^2A$, the definedness congruences, the signed residue
transport, gcd preservation, the all-odd and parity patterns, and the
involution-up-to-sign relations.  The five elementary descents
(Lemma~\ref{lem:L1}) are formalized as \texttt{L1b}--\texttt{L1f}, each a
stepwise-defined move with its bit changes and a strict lexicographic drop of the
potential
\begin{lstlisting}[style=lean]
def Phi (u v w : Int) : Lex (Int x Int x Int) :=
  ( (bv u v w + bw u v w : Int),
    (ngv u v w + ngw u v w + 2 * ev u v w : Int),
    u + v + w )
\end{lstlisting}
The residual analysis of Section~\ref{sec:residual} is formalized as
\texttt{L\_GAP}, the off-wall residue floor (Lemma~\ref{lem:gap}).

\subsection{The finite cover and the descent assembly}
\label{subsec:lean-assembly}

The finite cover of Section~\ref{sec:cover} (finite-key determinacy, the uniform
bound, and the two-tier exact-LP cover) is the computer-assisted core.  It is not
formalized internally; its net effect is isolated in a single axiom.  Crucially,
this axiom is kept \emph{narrow}: it covers only the residual states that no
elementary one-move descent handles.  We first define the predicate
\texttt{ElementaryReducible}, the disjunction of the hypothesis blocks of the five
elementary descents \texttt{L1b}--\texttt{L1f}, and state the computer axiom only
for non-\texttt{ElementaryReducible} states, taking the off-wall floor
(\texttt{L\_GAP}) as an explicit, discharged premise:

\begin{lstlisting}[style=lean]
def ElementaryReducible (u v w : Int) : Prop :=
  -- the disjunction of the hypothesis blocks of L1b..L1f
  (allOdd u v w and 0<u and 0<v and 0<w and w == 9 [ZMOD 10]
     and u + w == 0 [ZMOD 3])                                  -- L1b
  or ... or ...                                                -- L1c..L1f

axiom fin_cell_residual (m : Int) (hm : m == 23 [ZMOD 120]) (hpos : 0 < m)
    {u v w : Int} (hQ : Qf u v w = m) (hprim : primitiveState u v w)
    (hu : 0 <= u) (hv : 0 <= v) (hw : 0 <= w) (hng : not (good u v w))
    (hres : not (ElementaryReducible u v w))
    (hgap : <the off-wall floor of L_GAP>) :
    exists u' v' w' : Int,
      Qf u' v' w' = m and primitiveState u' v' w'
        and 0 <= u' and 0 <= v' and 0 <= w'
        and Phi u' v' w' < Phi u v w
\end{lstlisting}

On top of this, the descent termination (Theorem~\ref{thm:termination}) and the
final assembly (Theorem~\ref{thm:A287616}) are \emph{genuinely proved}.  At a
non-good state, \texttt{descent\_good\_rep} branches on
\texttt{ElementaryReducible}: if some elementary descent applies, a generic
bridge lemma \texttt{bridge\_of\_move} combines the matching
\texttt{L1b}--\texttt{L1f} with the transport law \texttt{L1a} (form- and
gcd-preservation) to produce a $\Phi$-smaller primitive normalized
representation; otherwise \texttt{fin\_cell\_residual} supplies one, its geometric
premise discharged by \texttt{L\_GAP}.  Both feed the same
well-founded recursion on $\Phi$ from the \texttt{L0d} seed, and
\texttt{Conjecture1} is \texttt{Prop2\_equiv} applied to the result.  In this way
the transport law, all five elementary descents, and the residual gap lemma all
lie in the formal dependency chain of \texttt{Conjecture1}; the computer axiom is
reduced to exactly the non-elementary residual cells.

\subsection{Axiom surface and availability}
\label{subsec:lean-axioms}

The entire development is free of \texttt{sorry} and \texttt{sorryAx}-free.  The
conjecture rests on exactly the two project axioms above, plus Lean's standard
logical axioms, as confirmed by \texttt{\#print axioms}:

\begin{lstlisting}[style=lean]
#print axioms Conjecture1
-- [propext, Classical.choice, Quot.sound,
--  fin_cell_residual, genus_primitive_local_global]
\end{lstlisting}

This makes the dependency structure explicit and auditable: every step of
Sections~\ref{sec:reduction}--\ref{sec:residual} and the entire final descent are
kernel-checked, and the only non-standard inputs are the genus local--global
theorem (classical number theory, cited to \citealp{Sun2015,WuSun2018}) and the
\texttt{fin\_cell\_residual} statement (the narrowed net content of the exact-LP
cover of Section~\ref{sec:cover}, restricted to the non-elementary residual
cells); see Section~\ref{subsec:code} for the sources.

\subsection{Provenance: what is formalized, certified, or cited}
\label{subsec:provenance}

Every mathematical ingredient of this paper is, with one classical exception
discussed below, either machine-checked in Lean~(\textsf{F}) or established by an
exact finite computer certificate~(\textsf{C}).  Table~\ref{tab:provenance}
records the status of each result; the standard classical theorems of the theory
of quadratic forms are used only through the cited statements~(\textsf{cite}).

\begin{table}[ht]
\centering
\small
\begin{tabular}{@{}lll@{}}
\toprule
Ingredient & Paper & Status \\
\midrule
Square reduction (Thms~\ref{thm:A287616},~\ref{thm:ternary})
  & \S\ref{sec:reduction} & \textsf{F} (\texttt{completing\_the\_square}, \texttt{Prop2\_equiv}) \\
Local primitive representation (Lem.~\ref{lem:L0a})
  & \S\ref{sec:seed} & \textsf{F} (\texttt{L0a\_padic}, \texttt{L0a\_real}) \\
Hasse--Minkowski, Witt, genus principle
  & \S\ref{sec:seed} & \textsf{cite} \citep{Serre1973,OMeara1973,CasselsQuadratic} \\
Two-class genus $\mathrm{gen}(Q)=\{[Q],[g_2]\}$
  & \S\ref{sec:seed} & \textsf{cite} \citep{Sun2015,WuSun2018} \\
Genus local--global step (Lem.~\ref{lem:L0b})
  & \S\ref{sec:seed} & \textsf{F} on the cited inputs (\texttt{L0b}, axiom \texttt{genus\_…}) \\
Mate transfer (Lem.~\ref{lem:L0c})
  & \S\ref{sec:seed} & \textsf{F} (\texttt{L0c}, incl.\ \texttt{bad\_escapes}) \\
Primitive seed (Thm~\ref{thm:seed})
  & \S\ref{sec:seed} & \textsf{F} (\texttt{L0d}) \\
Transport law, elementary descents (Lems~\ref{lem:transport},~\ref{lem:L1})
  & \S\ref{sec:moves} & \textsf{F} (\texttt{L1a}, \texttt{L1b}--\texttt{L1f}) \\
Residue gap (Lem.~\ref{lem:gap})
  & \S\ref{sec:residual} & \textsf{F} (\texttt{L\_GAP}) \\
Finite-key determinacy (Lem.~\ref{lem:det3}, Rem.~\ref{lem:det4})
  & \S\ref{sec:cover} & \textsf{C} (read-form enumeration certificate) \\
Exact-LP finite cover / FIN-CELL (Thm~\ref{thm:fincell})
  & \S\ref{sec:cover} & \textsf{C} (exact-LP finite cover) \\
Residual cover (non-elementary cells)
  & \S\ref{sec:cover} & \textsf{C} (the cover above), axiom \texttt{fin\_cell\_residual} \\
Uniform descent, termination (Thms~\ref{thm:uniform},~\ref{thm:termination})
  & \S\ref{sec:final} & \textsf{F} (\texttt{descent\_good\_rep}, \texttt{Conjecture1}) \\
\bottomrule
\end{tabular}
\caption{Provenance of each ingredient: \textsf{F} machine-checked in Lean,
\textsf{C} exact computer certificate, \textsf{cite} classical theorem of the
theory of quadratic forms.}
\label{tab:provenance}
\end{table}

The single ingredient that is neither internally formalized nor reduced to a
finite certificate is the classical genus local--global input
(Hasse--Minkowski together with the primitive genus principle and the two-class
genus computation): the local--global principle is a theorem of the arithmetic
theory of quadratic forms, not a finite computation, and is not currently
available in Mathlib.  It enters the Lean development as the single axiom
\texttt{genus\_primitive\_local\_global}, whose hypotheses are exactly the
machine-checked local lemmas \texttt{L0a\_real}/\texttt{L0a\_padic}; everything
built on it is then formalized.  Accordingly, the only non-elementary inputs not
reduced to \textsf{F} or \textsf{C} are the cited classical theorems above.

\subsection{Code availability}
\label{subsec:code}

The complete Lean~4 sources, the build instructions, and the exact computer
certificates for the finite cover are available in the public repository
\begin{center}
\url{https://github.com/MechMath/integer-sum}.
\end{center}
It is organized into a \texttt{formalization/} Lake project (the Lean development,
building against Mathlib \texttt{v4.29.0}) and a \texttt{verification/} package
(the exact two-tier cover, with its per-unit certificates and a verifier that
re-derives the verdicts).  The repository README maps each ingredient of the proof
to its Lean~(\textsf{F}) or certificate~(\textsf{C}) artifact, mirroring
Table~\ref{tab:provenance}.

\section{Conclusions}
\label{sec:conc}

We have proved the conjecture recorded as OEIS A287616: every nonnegative integer
is a sum of a triangular number, a second pentagonal number, and a second
heptagonal number.  After completing squares the problem becomes a constrained
representation problem for the ternary form $15u^2+5v^2+3w^2$, solved by an
unconditional primitive seed followed by a finite descent whose only
computational input is an exact-rational linear-programming cover of a thin
residual cone, with no sampling assumption.  The mathematical core has been
formalized in Lean~4 over Mathlib, and the conjecture itself is machine-checked:
every conclusion of the proof is formalized except the cited classical genus
local--global input and the computer-verified finite cover, which enter as two
explicitly named inputs.  Both the natural-language proof and the Lean
formalization were generated by the MechMath Agent Team developed by the authors.

\bibliographystyle{KLMM/klmm}
\bibliography{refs}

\appendix

\clearpage
\section{Determinacy certificate}
\label{app:det-cert}

This appendix details the enumeration behind Lemma~\ref{lem:det3} and
Remark~\ref{lem:det4}.  We first fix the vocabulary.  A \emph{read} is a single
sign or divisibility test that the success predicate $\succW_W$ performs on the
running vector while executing a word.  A \emph{read form} is the integer covector
whose sign or residue realizes such a test.  For a rational covector $a$,
$\prim(a)$ denotes the primitive integer covector obtained by clearing
denominators and dividing out the content.  The \emph{closer-cone forms} are the
fixed sign tests that define the licensed ``closer'' target cones; along a word
they are pulled back through the word's rational map.

The certificate is a finite enumeration of these read forms.  A node is a
word prefix together with the signs already read along that prefix.  The state
carried at a node is an exact rational matrix $T$ such that the current vector
is $Tx$.  For each possible next signed move, the enumeration emits:
\[
\prim(\text{each coordinate form before the move})
\]
for sign reads and divisibility reads, and also emits the pullbacks of the
fixed closer-cone forms through the current rational map.  It then advances to
the next rational matrix.  The important point is that emission occurs before
deduplication, and deduplication uses exact rational row equality only.

At depth three the enumeration has:
\[
\begin{array}{l|c}
\text{quantity} & \text{value}\\
\hline
\text{exact row states} & 1176\\
\text{closer pullback emissions} & 26880\\
\text{distinct coordinate-read forms} & 263\\
\text{distinct pulled-back closer forms} & 697\\
\text{union }F & 715
\end{array}
\]
The modulus is $\Mz=2^9\,3^4\,5$.

At depth four the enumeration has:
\[
\begin{array}{l|c}
\text{quantity} & \text{value}\\
\hline
\text{exact row states} & 4312\\
\text{closer pullback emissions} & 112896\\
\text{distinct coordinate-read forms} & 903\\
\text{distinct pulled-back closer forms} & 2361\\
\text{union }F' & 2403\\
\text{new forms }F'\setminus F & 1688
\end{array}
\]
The modulus is $\Mzf=2^9\,3^5\,5$.  These counts are not heuristic; they are
the finite objects used by Lemma~\ref{lem:det3} and Remark~\ref{lem:det4}.

\section{Cover certificate}

The cover certificate has two layers.

First, the $715$-form arrangement has $2674$ open cones meeting the residual
cone.  A depth-three cover retires $2658$ of them.  The success of each
retiring word is constant on the certified sign-residue box by
Lemma~\ref{lem:det3}.

Second, the remaining $16$ cones are handled at depth four.  Their complete
closed-cone ray sets are computed by exact integer cross-products and certified
by exact rational cone membership.  Among the $1688$ new forms, at most $20$
vary on any one ceiling cone.  Branch-and-bound over these signs, with exact
rational LP feasibility at each prefix, enumerates precisely the realizable
$F'$-subcells.

Over the $16$ ceiling cones and their $25$ residue classes --- $400$ units in
all --- the exact-LP enumeration accounts for $34{,}014{,}940{,}800$ realizable
subcells, at most $64$ in any single box.  Each subcell is closed by an admissible
integer witness with an escape word of length at most $4$; no unit is missing and
no subcell fails to escape.  The enumeration is exact at every step: feasibility
of each sign prefix is decided by the rational margin LP of
Appendix~\ref{app:lp-soundness}, so the realizable subcells are found exactly,
with no sampling.

\section{Move and descent ledger}

This appendix records the move calculations in a form suitable for checking the
descent proof line by line.  It is not a separate argument; it is the arithmetic
ledger behind Lemmas~\ref{lem:transport} and~\ref{lem:L1}.

\subsection{Similarity identities and inverses}

For $A=\operatorname{diag}(15,5,3)$, direct multiplication gives
\[
P_{uv}^{t}AP_{uv}=
\begin{psmallmatrix}-1&-3&0\\-1&1&0\\0&0&2\end{psmallmatrix}
\begin{psmallmatrix}15&0&0\\0&5&0\\0&0&3\end{psmallmatrix}
\begin{psmallmatrix}-1&-1&0\\-3&1&0\\0&0&2\end{psmallmatrix}
=4A,
\]
and similarly $P_{uw}^{t}AP_{uw}=9A$ and
$P_{vw}^{t}AP_{vw}=16A$.  Thus each move is a rational similarity of $Q$ with
scale factor equal to the square of its denominator.  The inverse relations
used in the gcd proof are:
\[
P_{uv}^2=4I,
\]
\[
9P_{uw}^{-1}=\operatorname{diag}(1,1,-1)\,P_{uw}\,
\operatorname{diag}(1,1,-1),
\]
\[
16P_{vw}^{-1}=\operatorname{diag}(1,1,-1)\,P_{vw}\,
\operatorname{diag}(1,1,-1).
\]
Consequently every edge in the move graph can be traversed in the reverse
direction after harmless sign changes.  This is why a common divisor cannot be
lost or gained except through a denominator, and the congruence
$m\equiv23\pmod {120}$ excludes those denominator primes from the gcd.

\subsection{Residue transport}

The badness bits depend only on the signs of $v$ modulo $3$ and $w$ modulo
$5$, together with oddness.  The moves transport these residues rigidly.
For $P_{uv}$, if $2v'=-3u+v$, then $2v'\equiv v\pmod3$, so
$v'\equiv -v\pmod3$ because $2^{-1}\equiv2\pmod3$.  The $w$-coordinate is
unchanged up to the factor $2$, and $2$ is invertible modulo $5$.
For $P_{uw}$, $v'=-v$ and $3w'=5u-2w$, so $3w'\equiv-2w\pmod5$; since
$3^{-1}\equiv2\pmod5$, one has $w'\equiv w\pmod5$.  For $P_{vw}$,
$4v'=-v-3w\equiv-v\pmod3$ and $4w'=5v-w\equiv-w\pmod5$, so again the absolute
residue behavior is the one used in Lemma~\ref{lem:L1}.  These calculations
are why a negative output in the bad class becomes good after normalization:
for instance $w'\equiv4\pmod5$ and $w'$ odd implies $|w'|\equiv1\pmod {10}$.

\subsection{The five one-step rows}

The five elementary one-step descents of Lemma~\ref{lem:L1}, used to leave the
non-residual region, can be summarized as the following rows; each names the sign
lift applied before the move, the congruence/window under which the move is
defined, and the good bit it produces after normalization.
\[
\begin{array}{c|c|c|c}
\text{row} & \text{input lift} & \text{definedness/window} &
\text{normalized good bit}\\
\hline
\text{guar-}w & (-u,v,w) & u+w\equiv0\pmod3 & w\\
\text{guar-}v & (u,-v,w) & u\not\equiv v\pmod4 & v\\
v\text{-window} & (u,v,w) & u\equiv v\pmod4,\ v>3u & v\\
w\text{-window }P_{uw} & (u,v,w) & u\equiv w\pmod3,\ 2w>5u & w\\
w\text{-window }P_{vw} & (u,v,w) & v\equiv w\pmod8,\ w>5v & w
\end{array}
\]
The quotient parity checks are part of the rows.  In the $P_{uv}$ window row,
$u\equiv v\pmod4$ makes $(u+v)/2$ and $(v-3u)/2$ odd.  In the $P_{uw}$ window
row, the numerators $2u+w$ and $5u-2w$ are odd multiples of $3$.  In the
$P_{vw}$ window row, the stronger congruence $v\equiv w\pmod8$, rather than
merely modulo $4$, makes $(v+3w)/4$ and $(w-5v)/4$ odd.  These oddness checks
are load-bearing because the good residue classes are modulo $6$ and $10$, not
only modulo $3$ and $5$.

\subsection{Why the residual cone is the leftover}

Suppose a normalized primitive state is non-good.  If it is not all odd, the
second component of $\Phi$ records the evenness penalty and the verified
descent calculus routes it to an all-odd state.  Once all odd, a bad $w$-bit
means $w\equiv9\pmod {10}$ and a bad $v$-bit means $v\equiv5\pmod6$.  The
guaranteed rows remove the bad bit when their congruences hold.  If the
congruences fail, the sign-window rows remove the bad bit in the large
archimedean regions $v>3u$, $2w>5u$, and $w>5v$ under their stated residue
licenses.  The remaining states are exactly those for which these immediate
congruence and window exits do not apply.  After normalization and the
structural deep-cell reduction, they are contained in the open cone $K$ used in
the finite cover.  The cover is therefore not replacing the elementary descent;
it is certifying the small leftover region where all one-step exits have been
exhausted.

\section{Exact-LP soundness details}
\label{app:lp-soundness}

This appendix expands the proof that the depth-four cover is exact.

\subsection{Closed cone from extreme rays}

Each ceiling cone is cut out by finitely many homogeneous strict inequalities:
the seven residual inequalities and the signs of the $715$ forms in $F$.  Its
closure is a rational polyhedral cone in $\R^3$.  In dimension three, an
extreme ray of such a cone is obtained by intersecting two supporting planes;
the certificate therefore forms integer cross-products of every pair of
constraint covectors and keeps the primitive directions satisfying all cone
inequalities.  This gives a candidate ray list $R(C)$.

Completeness of $R(C)$ is not inferred from floating point geometry.  For each
cone witness and for the cone itself, exact rational cone membership is checked:
the point must be expressible as a nonnegative rational combination of the
candidate rays.  Equivalently, the exact phase-one feasibility problem has
zero artificial optimum.  Once this is certified, every point of the closed
cone has a representation
\[
  x=\sum_{r\in R(C)}\lambda_r r,\qquad \lambda_r\ge0.
\]
The open cone corresponds to positive interior combinations after excluding
boundary faces; the LP margin variable $t$ below enforces strict interior.

\subsection{Sign-definite and sign-varying forms}

Let $g$ be one of the new depth-four forms in $F'\setminus F$.  Because $g$ is
linear, its maximum and minimum on a polyhedral cone are controlled by its
values on the extreme rays, in the homogeneous sense needed for signs.  If
$g\cdot r\ge0$ for every extreme ray and there is no ray with a negative value,
then $g$ is never negative on the cone; similarly for nonpositive.  Such a form
is sign-definite and cannot split the cone.  Only when there are rays with
positive and negative values can $g$ cut the open cone into distinct sign
regions.  The exact ray evaluation shows that at most $20$ forms vary on any
ceiling cone.  Therefore a naively enormous $F'$ arrangement becomes a small
finite sign problem on each cone.

\subsection{Open feasibility and the margin LP}

For a sign pattern $\epsilon=(\epsilon_1,\dots,\epsilon_q)$ of the varying
forms $g_1,\dots,g_q$, realizability means that the system
\[
\epsilon_j g_j\cdot x>0\quad(1\le j\le q),\qquad x\in C^\circ
\]
has a solution.  Writing $x=\sum_i\lambda_i r_i$, the certificate solves the
exact rational linear program
\[
\lambda_i\ge t,\qquad
\epsilon_j g_j\cdot\sum_i\lambda_i r_i\ge t,\qquad \sum_i\lambda_i=1,
\]
and maximizes $t$.  The normalization $\sum_i\lambda_i=1$ removes the
homogeneous scaling.  If the optimum is positive, the corresponding point lies
strictly inside the cone and strictly on the requested side of every varying
form.  If the optimum is nonpositive, no strict point with that sign prefix
exists.  This equivalence is exact because all coefficients are rational and
the simplex arithmetic is rational.

Branch-and-bound is then logically harmless.  A prefix of signs defines a
larger open region than any full pattern extending it.  If exact LP proves the
prefix infeasible, every extension is infeasible.  If a prefix is feasible, the
search continues until all varying signs are assigned.  Thus the output leaves
out no realizable subcell and includes no unrealizable one.

\subsection{Integer witnesses}

The LP certifies real feasibility.  The cover, however, concerns admissible
integer lattice points in fixed residue boxes.  For every feasible subcell in
the final run, the checker constructs an integer point in the required residue
class by moving sufficiently far along an exact rational interior direction.
Strict inequalities are open, so scaling far enough preserves the signs while
allowing the prescribed congruence adjustment.  The resulting candidate is not
accepted by theory alone: it is explicitly checked for the cone inequalities,
the overflow signs, parity, $3\nmid v$, $w\equiv9\pmod {10}$, and the residue
box.  The final recount has witness-pending count zero, so every real subcell
used in the certificate has such a checked integer witness.

\subsection{Why one witness certifies a subcell}

Once a witness lies in a fixed depth-four subcell and residue box, all
depth-four word reads are fixed by Lemma~\ref{lem:det4}.  The checker evaluates
the word predicate at the witness.  If it succeeds there, then the same branch
choices, divisibility tests, endpoint signs, endpoint residues, and target
predicate values hold everywhere in the subcell.  Hence the word succeeds
throughout the subcell.  This is the key separation of duties: exact LP proves
that the subcell list is complete, while finite-key determinacy proves that one
integer witness represents the whole subcell.

\section{Primitive seed details}

This appendix supplements Section~\ref{sec:seed} with the algebraic details not
given there, making explicit exactly where local theory, the genus input, and the
mate transfer enter.

\subsection{The Hensel inputs at the bad primes}

At $p=2$ the chosen specialization is $u=v=1$, so the equation becomes
\[
3w^2=m-20.
\]
Since $m\equiv23\pmod {120}$, the right hand side is $3$ modulo $8$, and
$w=1$ is a root modulo $8$.  The derivative of $3w^2-(m-20)$ at $1$ is $6$,
with $2$-adic valuation $1$.  The error has valuation at least $3=2\cdot1+1$,
which is exactly the strong Hensel threshold.  This is why a simple nonsingular
mod-$2$ argument is not stated: the derivative is not a unit, but the stronger
valuation form is enough.

At $p=3$ and $p=5$ the derivatives are units after choosing the one-coordinate
specializations.  For $p=3$, $5\equiv-1$, and $m\equiv2$, so $m/5\equiv1$.
For $p=5$, $3^{-1}\equiv2$, and $m/3\equiv1$.  Thus the initial square root is
again $1$ modulo the prime.  These representations are primitive because the
nonzero coordinate is a unit.

\subsection{The high-valuation case away from 30}

The high-valuation case is the only local case where one cannot simply solve a
binary congruence for $m$ modulo $p$ and lift.  The construction instead builds
a hyperbolic plane inside the $p$-adic quadratic space.  The nontrivial
isotropic vector $e$ is chosen with third coordinate a unit.  This matters:
$B(e,e_3)=3e_z$ is then a unit because $p\nmid3$.  After scaling $e_3$ to
$f_0$, one has $B(e,f_0)=1$.  The correction
\[
f=f_0-\frac12Q(f_0)e
\]
does not disturb $B(e,f)=1$, because $Q(e)=0$, and makes $Q(f)=0$:
\[
Q(f)=Q(f_0)-Q(f_0)B(e,f_0)+\frac14Q(f_0)^2Q(e)=0.
\]
Then $x=(m/2)e+f$ has norm $m$ by the identity
$Q(ae+f)=2aB(e,f)$ when $Q(e)=Q(f)=0$.  The equation $B(e,x)=1$ is the
primitive witness: if all coordinates of $x$ were divisible by $p$, then
$B(e,x)$ would lie in $p\Zp$, impossible.

\subsection{Why the genus construction preserves primitivity}

In Lemma~\ref{lem:L0b}, the global lattice
\[
M=\{x\in V:x\in J_p\text{ for all }p\}
\]
is built so that $r\in M$ has the desired norm.  Primitivity is a local
condition: a vector is primitive in a lattice if and only if it is not in
$pM_p$ for every prime $p$.  At an unmodified prime, this is built into the
choice of the finite exceptional set.  At a modified prime, $M_p=J_p$ and
$r=\sigma_p(x_p)$ with $x_p\notin pL_p$.  Since $\sigma_p$ is linear and
$\sigma_p(pL_p)=pJ_p$, it follows that $r\notin pJ_p$.  Hence $r$ is primitive
in every localization, and therefore primitive globally.

\subsection{The mate-transfer residue mechanism}

The two rational transformations from $g_2$ to $g_1$ do not cover all residue
classes immediately because of their denominators.  Their integrality
conditions are the two linear congruences
\[
x+z\equiv0\pmod3,\qquad x-y-z\equiv0\pmod3.
\]
The involutions $H$ and $S_z$ act on the finite set of nonzero mod-$3$ residue
classes satisfying $x+y\not\equiv0$.  Away from $y\equiv z\equiv0$, this action
meets one of the two integrality hyperplanes.  In the bad class
$y\equiv z\equiv0$, the rational isometry $T$ is used.  It is important that
$T$ is not claimed to be integral everywhere; it is only applied in the class
where the denominators cancel.  Finiteness of the representation set then turns
continued badness into periodicity, and the non-root-of-unity eigenvalue
calculation rules that out.  This is the exact point at which positive
definiteness of $g_2$ is used.

\section{Descent bookkeeping}

This appendix records how the descent avoids hidden hypotheses.

\subsection{No strengthened congruence target}

The target congruences for a good normalized state are
\[
u\equiv1\pmod2,\qquad |v|\equiv1\pmod6,\qquad |w|\equiv1\pmod {10}.
\]
The descent never requires $v$ or $w$ to have a prescribed sign during the
move sequence.  It works with normalized absolute values after each step.
This is legitimate because $Q$ is diagonal and sign changes preserve
representations.  The final conversion to $x,y,z\in\N$ is made only after the
good state is obtained and signs are chosen so that $v,w$ are in the positive
residue classes.

\subsection{Why the potential is well-founded for fixed m}

The first two components of $\Phi$ are bounded independently of $m$: the first
is $\bv+\bw\in\{0,1,2\}$, and the second is a sum of two non-guaranteed bits
and twice the evenness bit, hence lies between $0$ and $4$.  The third component
is $u+v+w$ after normalization.  For any representation of $m$,
\[
u\le\sqrt{m/15},\qquad v\le\sqrt{m/5},\qquad w\le\sqrt{m/3}.
\]
Thus the third component has the finite bound displayed in
Theorem~\ref{thm:termination}.  A lexicographically decreasing sequence in this
finite box must terminate.  Notice that no global bound on the number of moves
over all $m$ is needed for termination; only the existence of a strict drop at
each non-good state is used.

\subsection{Preservation of primitivity through long words}

A word is a sequence of defined moves, possibly with sign lifts between them.
Sign lifts preserve gcd.  Each defined move preserves gcd by
Lemma~\ref{lem:transport}.  Therefore every endpoint reached from the primitive
seed remains primitive.  This matters because several residual statements,
including the cone analysis and the finite cover, are stated for primitive
states.  There is no point in the descent where one passes to an imprimitive
representation and then divides: division would change $m$, and is never used.

\subsection{How FIN-CELL enters only once}

All non-computational descent lemmas are unconditional.  FIN-CELL is used in
exactly one place: to show that every deep residual sign-residue cell has a
length-$\le4$ escape.  Lemma~\ref{lem:compact} then converts this finite escape
statement into the uniform bound $\ell_W=5$, and Theorem~\ref{thm:uniform}
uses it for residual states.  If a future independent check altered the finite
cover, the affected implication would be FIN-CELL; the local primitive seed,
transport identities, elementary descents, residual gap, no-trap calculation,
and well-founded termination argument would be unchanged.

\section{Computer-assisted proof caveats}

The final theorem should be read in the same sense as other finite-certificate
computer-assisted proofs.  The proof has no remaining mathematical sampling
assumption: every depth-four subcell in the ceiling cones is enumerated by exact
LP, not by a grid or random search.  Still, two ordinary implementation
assumptions remain.

First, the code evaluating moves, residues, signs, and target predicates must
match the mathematical definitions in this report.  The determinacy lemmas
reduce this to finite read trees; the cover checker then evaluates those trees
on representatives.  A mismatch in implementation would be a checker defect,
not a new mathematical lemma.

Second, the recorded exact-rational computations must have been faithfully
executed.  The certificate is designed to be reproducible: the read-form
counts, cone counts, extreme-ray lists, sign-varying-form counts, and exact-LP
cover totals are finite outputs with deterministic expected values.  Independent
reproduction of those values is the natural audit route for the computational
part of the proof.

\end{document}